\documentclass[12pt]{article}

\usepackage{amsmath,amssymb,amstext,amsthm} 

\theoremstyle{plain}
\newtheorem{theorem}{Theorem}[section]
\newtheorem{lemma}[theorem]{Lemma}
\newtheorem{proposition}[theorem]{Proposition}
\newtheorem{corollary}[theorem]{Corollary}

\theoremstyle{definition}
\newtheorem*{definition}{Definition}

\usepackage{tikz}
\tikzstyle{graphnode}=[circle, draw, fill=black!30,
                        inner sep=0pt, minimum width=4pt]
\tikzstyle{unmatched}=[graphnode,fill=black!0]
\tikzstyle{matched}=[graphnode,fill=black!100]
\tikzstyle{matching} = [ultra thick]

\title{Cops and Robber Game with a Fast Robber \\ on Interval, Chordal, and Planar Graphs}

\author{Abbas Mehrabian\thanks{\texttt{amehrabi@math.uwaterloo.ca}} \\ 
{\small Department of Combinatorics and Optimization} \\ {\small University of Waterloo}}

\date{}

\begin{document}

\maketitle

\begin{abstract}
We consider a variant of the Cops and Robber game, 
introduced by Fomin, Golovach, Kratochv{\'{\i}}l,
in which the robber has unbounded speed,
i.e.~can take any path from her vertex in her turn,
but she is not allowed to pass through a vertex occupied by a cop.
We study this game on interval graphs, chordal graphs, planar graphs, and hypercube graphs.
Let $c_{\infty}(G)$ denote the number of cops needed to capture the robber in graph $G$
in this variant.
We show that if $G$ is an interval graph, then  $c_{\infty}(G) = O(\sqrt {|V(G)|})$,
and we give a  polynomial-time 3-approximation algorithm for finding $c_{\infty}(G)$ in interval graphs.
We prove that for every $n$
there exists an $n$-vertex chordal graph $G$ with $c_{\infty}(G) = \Omega(n / \log n)$.
Let $tw(G)$ and $\Delta(G)$ denote the treewidth and the maximum degree of $G$, respectively.
We prove that for every $G$,
$tw(G) + 1 \leq (\Delta(G) + 1) c_{\infty}(G)$.
Using this lower bound for $c_{\infty}(G)$, we show two things.
The first is that if $G$ is a planar graph
(or more generally, if $G$ does not have a fixed apex graph as a minor),
then $c_{\infty}(G) = \Theta(tw(G))$.
This immediately leads to an $O(1)$-approximation algorithm for computing $c_{\infty}$ for planar graphs.
The second is that if $G$ is the $m$-hypercube graph, then
there exist constants $\eta_1,\eta_2>0$ such that
$\eta_1 2^m / (m\sqrt m) \leq c_{\infty}(G) \leq \eta_2 2^m / m$.

\end{abstract}

\section{Introduction}
\label{chp:introduction}
The game of \emph{Cops and Robber} is a perfect information game,
played in a graph $G$.
The players are a set of cops and a robber.
Initially, the cops are placed at vertices
of their choice in $G$ (where more than one cop can be placed at a vertex).
Then the robber, being
fully aware of the cops' placement, positions herself at one of the vertices of $G$.
Then the cops and
the robber move in alternate rounds, with the cops moving first;
however, players are permitted to
remain stationary in their turn if they wish.
The players use the edges of $G$ to move from vertex to vertex.
The cops win, and the game ends, if eventually a cop moves to the vertex currently occupied
by the robber; otherwise, i.e.~if the robber can elude the cops forever, the robber wins.

This game was defined (for one cop) by Winkler and Nowakowski~\cite{game_definition_1}
and Quilliot~\cite{game_definition_2}, and has been studied extensively.
For a survey of results on this game, see the survey by Hahn~\cite{survey_hahn}.
The famous open question in this area is Meyniel's conjecture, published by Frankl~\cite{large_girth},
which states that for every connected graph on $n$ vertices, $O(\sqrt n)$ cops are sufficient to capture the robber.
The best result so far is that
$$n2^{-\left(1-o(1)\right)\sqrt{\log_2 n}}$$
cops are sufficient to capture the robber.
This was proved independently by Lu and Peng~\cite{lu_peng},
and Scott and Sudakov~\cite{scott_sudakov}.

One interesting fact about the Cops and Robber game is that,
many scholars have studied the game,
and yet it is not really well understood:
although the upper bound  $O(\sqrt n)$
was conjectured in 1987,
no upper bound better than $n^{1-o(1)}$ has been proved since then.
As an another example, no efficient approximation algorithm
for finding the number of cops needed to capture the robber in a given graph has been developed.

One might try to change the rules of the game a little in order to get a more approachable problem,
and/or to understand what property of the original game causes the difficulty.
Several variations of the game has been studied,
by changing the rules slightly,
e.g.~by limiting the visibility of the cops~\cite{regular_visible_robber},
by limiting the visibility of both players~\cite{randomized_local_visibility},
by changing the definition of capturing~\cite{shooting_cop},
or by allowing the players to move only in a certain direction along each edge~\cite{variations}.

The approach chosen 
by Fomin, Golovach, Kratochv\'{\i}l, Nisse, and Suchan~\cite{fast_robber_first_journal} is
to allow the robber move faster than the cops.
Inspired by their work,
in this paper we let the robber take \emph{any path} from her
current position in her turn,
but she is not allowed to pass through a vertex occupied by a cop.
The parameter of interest is the \emph{cop number} of $G$,
which is defined as the minimum number of cops needed to ensure that the cops can win.
We denote the cop number of $G$ by $c_{\infty}(G)$,
in which the $\infty$ at the subscript indicates that the robber has unbounded speed.
A nice fact about this variation is its analogy with
the so-called Helicopter Cops and Robber game
(defined in~\cite{helicopter}, see Section~\ref{chp:treewidth} for the definition).
This is a real-time pursuit-evasion game with a robber of unbounded speed,
for which Seymour and Thomas have shown that the number of cops needed equals the treewidth
of the graph (which is a fairly well understood parameter)
plus one~\cite{helicopter}.
Thus one may hope to get good bounds for the cop number in terms of treewidth by
relating our variant of the Cops and Robber game
and the Helicopter Cops and Robber game,
and this is what we do in Section~\ref{chp:treewidth}.
However, one should not be deceived by this analogy;
the cop number can be arbitrarily smaller than the treewidth:
any graph with small domination number and large treewidth
(e.g.,~a complete graph) is such an example.
Therefore, this paper can also be regarded as an attempt to find connections between
the original Cops and Robber game
and the Helicopter Cops and Robber game
by studying an in-between game.
Nevertheless, treewidth is closely connected with cop number,
and in Section~\ref{chp:treewidth}, which is completely devoted to this connection,
we prove bounds for cop number in terms of treewidth.
In Sections~\ref{chp:planar},~\ref{chp:hypercube}~and~\ref{chp:large},
we will see three applications of these bounds.
Tree and path decompositions arise naturally and are important
when studying the cop number, and
the idea of several proofs in Sections~\ref{chp:interval}~and~\ref{chp:chordal}
is based on them
(although they do not appear explicitly in these sections).

This variant was first studied
by Fomin, Golovach, Kratochv{\'{\i}}l~\cite{fast_robber_first}.
They proved that computing $c_{\infty}(G)$ is an NP-hard problem, even if $G$
is a split graph.
(A \emph{split graph} is a graph whose vertex set can be partitioned into a clique and an independent set.)
This variant was further studied by Frieze, Krivelevich and Loh~\cite{variations},
where the authors' approach is based on expansion.
In~\cite{variations}, it is shown that for each $n$, there exists a connected graph on $n$ vertices with cop number $\Theta(n)$.

Let $G$ be a connected graph with $n$ vertices.
In Section~\ref{chp:interval},
we show that if $G$ is an interval graph, then  $c_{\infty}(G) = O(\sqrt n)$
and provide examples for which this bound is tight.
We also give a  polynomial-time 3-approximation algorithm for finding $c_{\infty}(G)$.
In Section~\ref{chp:chordal}, we prove that for every $n$
there exists a chordal graph $G$ with $c_{\infty}(G) = \Omega(n / \log n)$.
Let $tw(G)$ and $\Delta(G)$ denote the treewidth and the maximum degree of $G$, respectively.
In Section~\ref{chp:treewidth}, we prove that for every $G$,
$$\frac{tw(G) + 1 }{\Delta(G) + 1} \leq c_{\infty}(G) \leq tw(G) + 1,$$
and provide examples for which these bounds are tight.
We will see applications of this result in the three subsequent sections.
In Section~\ref{chp:planar}, we show
that if $G$ is a planar graph
(or more generally, if $G$ does not have a fixed apex graph $H$ as a minor),
then $c_{\infty}(G) = \Theta(tw(G))$.
This immediately leads to an $O(1)$-approximation algorithm for computing the cop number of planar
(in general, apex-minor-free) graphs.
In Section~\ref{chp:hypercube}, we show that if $G$ is the Cartesian product of $m$ copies of $K_k$,
then there exist positive constants $\kappa_1,\kappa_2$ such that
$$\frac{\kappa_1 n}{k m\sqrt m} \leq c_{\infty}(G) \leq \min\left\{ \frac{n}{k},\frac{\kappa_2 n}{\sqrt m}\right\}.$$
Moreover, if $G$ is the $m$-hypercube graph, then
there exist constants $\eta_1,\eta_2>0$ such that
$$\frac{\eta_1 n}{m\sqrt m} \leq c_{\infty}(G) \leq \frac{\eta_2 n}{m}.$$
In Section~\ref{chp:large} we give a short proof for
the fact that for each $n$, there exists a connected graph on $n$ vertices with cop number $\Theta(n)$,
which is proved in~\cite{variations} using other ideas.
We conclude with some open problems in Section~\ref{chp:future}.

\subsection{Preliminaries and notation}
Let $G$ be the graph in which the game is played.
In this paper $G$ is always finite,
and $n$ always denotes the number of vertices of $G$.
We will assume that $G$ is simple, because deleting multiple edges or loops
does not affect the set of possible moves of the players.
We consider only connected graphs, since the cop number of a disconnected graph
obviously equals the sum of the cop numbers for each connected component.
As we are only interested in studying the cop number,
we may assume without loss of generality
that the cops choose vertices of our choice in the beginning,
since they can move to the vertices of their choice later.

For a subset $A$ of vertices, the \emph{neighbourhood} of $A$, written $N(A)$,
is the set of vertices that have a neighbour in $A$,
and the \emph{closed neighbourhood} of $A$, written $\overline{N}(A)$, is the union $A \cup N(A)$.
If $A=\{v\}$ then we may write $N(v)$ and $\overline{N}(v)$ instead of $N(A)$ and $\overline{N}(A)$, respectively.
A \emph{dominating set} is a subset $A$ of vertices with $V(G) = \overline{N}(A)$,
and the \emph{domination number} of $G$ is the minimum size of a dominating set of $G$.
The subgraph induced by $A$ is written $G[A]$, and the subgraph induced by $V(G) - A$ is written $G-A$.

\section{Interval Graphs}
\label{chp:interval}

Graph $G$ is an \emph{interval graph}
if there is a correspondence between its vertices and a set of closed intervals on the real line,
such that two vertices are adjacent in $G$ if and only if their corresponding intervals intersect.
Let $G$ be an interval graph.
Fomin~et~al.~\cite{fast_robber_first} proved that
if the robber has fixed speed, the number of cops needed to capture the robber can be computed in polynomial time.
The complexity of computing $c_{\infty}(G)$ was left open in~\cite{fast_robber_first}.
As a partial answer, in this section we prove that this problem is 3-approximable.
We also prove that $c_{\infty}(G) = O(\sqrt n)$
for all connected interval graphs $G$,
and provide examples for which this bound is tight.

\begin{definition}[$k$-wide]
For a subgraph $H$ of $G$, say $H$ is \emph{$k$-wide} if
\begin{itemize}
\item[(i)] $H$ is $k$-connected, and

\item[(ii)] for any $S \subseteq V(G)$ with $|S|<k$ we have $V(H) \not \subseteq \overline{N}(S)$.
\end{itemize}
\end{definition}

\begin{lemma}
If $G$ has a $k$-wide subgraph $H$ then $c_{\infty}(G) \geq k$.
\label{lower_wide}
\end{lemma}

\begin{proof}
Say a cop \emph{controls} a vertex $u$ if the cop is at $u$ or at an adjacent vertex.
Suppose that there are less than $k$ cops in the game, and they initially start at a subset $S$ of vertices.
By condition (ii), there is a vertex $v \in V(H) \setminus \overline{N}(S)$, i.e.~$v$ is controlled by no cop.
The robber starts at $v$, and will always remain in $H$.
After each move of the cops, the set of vertices occupied by them has size less than $k$.
Hence by condition (ii), there exists a vertex $x$ of $H$ that is not controlled by any of the cops.
By condition (i), $H$ is $k$-connected, so as the robber is currently in $H$,
and the number of cops is less than $k$, there is a cop-free path to $x$.
The robber moves there and will not be captured in the next round.
Since she can elude forever by using this strategy, at least $k$ cops are needed to capture her.
\end{proof}

For the rest of this section, $G$ is a connected interval graph.
Consider a set of closed intervals whose intersection graph is $G$, and denote by $I_v$ the interval corresponding to the vertex $v \in V(G)$.
We may assume without loss of generality that none of the intervals have zero length.
Such a representation can be found in polynomial time (see~\cite{interval_detection} for instance).
Let $x_1 < x_2 < \dots < x_{l+1}$ be the set of distinct endpoints of the intervals,
and let $y_1,y_2,\dots,y_l$ be points satisfying $x_i < y_i < x_{i+1}$ for $1\leq i \leq l$.
Also, define $V_i = \{v \in V(G) : y_i \in I_v\}$ for all $1\leq i \leq l$.
It is clear that
each $G[V_i]$ is a clique for $1\leq i \leq l$
(recall that $G[V_i]$ denotes the subgraph induced by $V_i$).
Furthermore, $l\leq 2n$ and the sets $V_1,\dots,V_l$ cover the vertices of $G$.

We say $A$ is a \emph{cut-set} of $G$ if $G-A$ has more connected components than $G$.

\begin{lemma}
\label{interval_cutsets}
Every minimal cut-set $X$ of $G$ is one of the $V_i$'s.
Moreover, if $X=V_i$ is a cut-set, then for each $u_1 \in V_{i_1} \setminus X$ and $u_2 \in V_{i_2} \setminus X$ satisfying $i_1<i<i_2$,
$u_1$ and $u_2$ lie in different components of $G - X$.
\end{lemma}

\begin{proof}
For an index $1\leq i \leq l$, say point $y_i$ is a \emph{cut-point} if there exists a vertex $v\in V(G)$ with both endpoints of $I_v$ lying strictly on the left of $y_i$, and also a vertex $v'\in V(G)$ with both endpoints of $I_{v'}$ lying strictly on the right of $y_i$.
If $y_i$ is a cut-point then clearly $V_i$ is a cut-set of $G$.

Now, let $X$ be a minimal cut-set of $G$.
Let $u_1,u_2$ be vertices in different components of $G- X$, with $I_{u_1} = [x_a,x_b], I_{u_2}=[x_c,x_d]$, and assume by symmetry that $a < b<c< d$.
For each $i$ with $b\leq i \leq c-1$, $y_i$ is a cut-point.
If for all of the $i$'s in this range, there was a vertex $v_i \in V_i \setminus X$, then $u_1 v_b v_{b+1} \dots v_{c-1} u_2$ would be a $(u_1,u_2)$-path in $G- X$.
As such a path does not exist, there is an $i$ in this range such that $V_i \subseteq X$.
But then $V_i$ is a cut-set of $G$, hence $X=V_i$.

For the second statement, let $X=V_i$ be a cut-set, $u_1 \in V_{i_1} \setminus X$ and $u_2 \in V_{i_2} \setminus X$ such that $i_1<i<i_2$.
Let $I_{u_1} = [x_a,x_b], I_{u_2}=[x_c,x_d]$, and so $x_a< x_b < y_i < x_c< x_d$.
Every $(u_1,u_2)$-path contains a vertex whose corresponding interval contains $y_i$, but all such vertices are in $X$.
Hence there is no $(u_1,u_2)$-path in $G - X$.
\end{proof}

\begin{definition}[$G[a,b{]}$, interval subgraph, $w(G)$]
We write $G[a,b]$ for the subgraph induced by $\bigcup_{a\leq i \leq b} V_i$ (for $1\leq a \leq b \leq l$), and we call each of these an \emph{interval subgraph}.
Let $w(G)$ be the maximum number $M$ such that $G$ has an $M$-wide interval subgraph.
\end{definition}

\begin{lemma}
\label{interval_wG}
It is possible to compute $w(G)$ in polynomial time.
\end{lemma}
\begin{proof}
Fix an interval subgraph $G[a,b]$.
It is easy to see that there is an $S \subseteq V(G)$ with $V(G[a,b]) \subseteq \overline{N}(S)$
if and only if the domination number of $G[a,b]$ is at most $|S|$,
that is,
if there is a set of $|S|$ vertices of $G$ dominating the vertices of $G[a,b]$,
then there exists such a set inside $G[a,b]$.
Moreover, $G[a,b]$ is an interval graph so its domination number can be found in polynomial time (using a greedy algorithm).
The connectivity of $G[a,b]$ can also be computed in polynomial time (see \cite{finding_connectivity} for example).
Therefore, the largest $M$ such that $G[a,b]$ is $M$-wide can be computed in polynomial time.
Recall that $w(G)$ is the maximum number $M$ such that $G$ has an $M$-wide interval subgraph.
The total number of interval subgraphs is $O(l^2)=O(n^2)$, so $w(G)$ can be computed in polynomial time.
\end{proof}

The following lemma gives an appropriate upper bound for $c_{\infty}(G)$.

\begin{lemma}
We have $c_{\infty}(G) \leq 3 w(G)$.
\label{interval_upperbound}
\end{lemma}

\begin{proof}
We just need to give a strategy for $3w(G)$ cops to capture the robber.
Let $M = w(G)$.
There are three teams of cops, each of size $M$.
At the beginning the first team starts at a vertex in $V_1$, the second team starts at a vertex in $V_l$, and the third team starts at an arbitrary vertex.
Suppose that the robber starts at a vertex $r$.
The cops' strategy consists of several (at most $l$) phases, in each of which they reduce the free space of the robber.
The following invariant is true at the start of each phase:
the $j$-th team ($j=1,2$) is in a subset $X_j \subseteq V_{i_j}$ such that they block the robber from escaping $G[i_1,i_2]$.

Note that during this phase, if the robber goes to a vertex in $V_{i_1}\cup V_{i_2}$ then she will be captured immediately by the first or second team (recall that each $G[V_i]$ is a clique).
If $i_2 \leq i_1+1$ then she should move to a vertex in $V_{i_1}\cup V_{i_2}$ and will be captured immediately, so assume that $i_2 > i_1+1$.
Since $G[i_1+1, i_2-1]$ is not $(M+1)$-wide,
either $G[i_1+1, i_2-1]$ has a minimal cut-set $X$ of size at most $M$,
or $G[i_1+1, i_2-1]$ has a dominating set $X$ of size at most $M$.

In the second case, the third team moves to $X$
(while the first and second teams stay still and block the robber from escaping $G[i_1,i_2]$),
and the robber will be captured in the next round.

In the first case, the third team moves to $X$, and suppose that $X = V_{i_3}$ (by Lemma~\ref{interval_cutsets}, $X$ is of this form).
Suppose that the robber moves to $r$ right after the third team has settled in $X$ and $j$ be an index such that $r \in V_j$.
If $j=i_3$ then the third team immediately captures her (since $G[V_{i_3}]$ is a clique), so assume, by symmetry, that $i_1 < j < i_3$.
Now, the first team together with the third team block the robber from escaping the subgraph $G[i_1,i_3]$ (by the second statement in Lemma~\ref{interval_cutsets}).
The second and third team switch roles and this phase finishes.
Note that $i_3-i_1 < i_2-i_1$ so the total number of phases is not larger than $l$.
\end{proof}

\begin{theorem}
\label{thm:intervals_approximation}
There exists a polynomial-time 3-approximation algorithm for computing $c_{\infty}(G)$ when $G$ is an interval graph.
\end{theorem}
\begin{proof}
Given $G$, the sequence $(V_1,V_2,\dots,V_l)$ can be found efficiently.
Then $w(G)$ can be computed in polynomial time by Lemma~\ref{interval_wG}.
The value $3w(G)$ is a 3-approximation for $c_{\infty}(G)$ by Lemmas~\ref{lower_wide}~and~\ref{interval_upperbound}.
\end{proof}

Next we prove that $c_{\infty}(G) = O(\sqrt{n})$.
Before doing so, we note that this bound is tight:
let $G$ be the strong product of the path on $3m$ vertices and the complete graph on $m$ vertices.
That is,
$$V(G) = \{1,2,\dots,3m\} \times \{1,2,\dots,m\},$$ and
$$\{ (i,j), (k,l)\} \in E(G) \ \mathrm{if}\ (i,j) \neq (k,l)\ \mathrm{and}\ |i-k| \leq 1.$$
Then $G$ is an interval graph with $3m^2$ vertices, and is $m$-wide itself, hence $$c_{\infty}(G) \geq m = \Omega(\sqrt{|V(G)|}).$$

We will need a lemma about minimum dominating sets in interval graphs,
which may not be the best possible, but suffices for our purposes.

\begin{lemma}
\label{lem:mindominating}
Let $A$ be a minimum dominating set of $G$.
Every vertex $v\in A$ is adjacent to at most two vertices of $A$,
and every vertex $v\notin A$ is adjacent to at most five vertices of $A$.
\end{lemma}

\begin{proof}
Let $I_v = [x,y]$ be the interval corresponding to vertex $v$.
First, let $v\in A$.
If there is a vertex $u\in A$ whose corresponding interval contains $I_v$,
then $\overline{N} (v) \subseteq \overline{N} (u)$, which contradicts the minimality of $A$.
If there is a vertex $u\in A$ whose corresponding interval is contained in $I_v$,
then $\overline{N} (u) \subseteq \overline{N} (v)$, which contradicts the minimality of $A$.
So for every $u \in A$ that is adjacent to $v$,
the interval corresponding to $u$ contains exactly one of $x$ and $y$.
If there are two distinct vertices in $N(v) \cap A$ whose corresponding intervals contain $x$,
then one can remove one of them (the one whose left-end-point of the corresponding interval is more to the right) from $A$,
and still have a dominating set, which contradicts the minimality of $A$.
Thus there exists at most one vertex in $N(v) \cap A$ whose corresponding interval contains $x$.
Similarly, there exists at most one vertex in $N(v) \cap A$ whose corresponding interval contains $y$,
so $|N(v) \cap A| \leq 2$.

Second, let $v\notin A$.
If there is a vertex $u\in A$ whose corresponding interval contains $I_v$,
then since $u$ is adjacent to at most two vertices of $A$,
$v$ is adjacent to at most three vertices of $A$.
So we may assume that this is  not the case.
If there are two distinct $u_1,u_2\in A$ whose corresponding intervals are contained in $I_v$,
then $\overline{N} (u_1) \cup \overline{N} (u_2) \subseteq \overline{N} (v)$,
which contradicts the minimality of $A$.
By an argument similar to the one in the previous case,
it can be shown that there are at most two distinct vertices in $A\cap N(v)$
whose corresponding intervals contain $x$.
Similarly, there are at most two distinct vertices in $A\cap N(v)$
whose corresponding intervals contain $y$.
Thus, $v$ is adjacent to at most five vertices of $A$.
\end{proof}

\begin{theorem}
Let $G$ be a connected interval graph with $n$ vertices.
Then $c_{\infty}(G) = O(\sqrt n)$.
\end{theorem}
\begin{proof}
By Lemma~\ref{interval_upperbound} it is enough to show that $w(G) = O(\sqrt n)$.
Let $G[a,b]$ be an arbitrary interval subgraph of $G$.
We just need to prove that $G[a,b]$ is not $(\sqrt{5n}+3)$-wide.
Choose two arbitrary vertices $u_a \in V_a, u_b \in V_b$.
Let $a'$ be the smallest index in $\{a,a+1,\dots,b\}$ with $u_a \notin V_{a'}$,
and $b'$ be the largest index in $\{a,\dots,b\}$ with $u_b \notin V_{b'}$.
If either of these indices does not exist or $a'> b'$, then $\{u_a,u_b\}$ is a dominating set for $G[a,b]$, so it is not $(\sqrt{5n}+3)$-wide.
So, we may assume that $a'$ and $b'$ exist.

Consider the graph $G[a',b']$.
Let $n_1$ be its number of vertices, $T$ be a minimum dominating set for it, and $\delta$ be its minimum degree.
Let $t = |T|$.
Note that $T\cup \{u_a,u_b\}$ is a dominating set for $G[a,b]$, so the domination number of $G[a,b]$ is  at most $t+2$.
Moreover, $G[a',b']$ is an interval graph,
so by Lemma~\ref{lem:mindominating},
every vertex $v \in V\left(G[a',b']\right) \setminus T$ is adjacent to at most five vertices of $T$,
and every vertex $v\in T$ is adjacent to at most two vertices of $T$,
hence (denoting the degree of $u$ in $G[a',b']$ by $deg(u)$) we have
$$t(\delta+1) \leq \sum_{u\in T} \left(deg(u)+1\right) \leq 5n_1 \leq 5n,$$
so $\min\{t, \delta +1\} \leq \sqrt{5n}$.

If $t \leq \sqrt{5n}$ then the domination number of $G[a,b]$ is at most $\sqrt{5n}+2$ so it is not $(\sqrt{5n}+3)$-wide.
So we may assume that $\delta+1 \leq \sqrt{5n}$.
Let $u$ be a vertex of minimum degree in $G[a',b']$, which is contained in some $V_i, a' \leq i \leq b'$.
Thus $|V_i| \leq \delta+1 \leq \sqrt{5n}$ and $V_i$ is a cut-set in $G[a,b]$ (as it separates $u_a,u_b$), so $G[a,b]$ is not $(\sqrt{5n}+1)$-connected, and hence not $(\sqrt{5n}+3)$-wide.
\end{proof}


\section{Chordal Graphs}
\label{chp:chordal}

A \emph{chordal} graph is a graph that does not have an induced cycle with more than 3 vertices.
Note that any interval graph is chordal.
It is well known that in the original Cops and Robber game,
a single cop can capture the robber in a chordal graph
(an easy way to see this is by considering a tree decomposition
of $G$ in which each bag induces a clique in $G$).
However, when the robber has unbounded speed the situation is quite different.
In this section we prove that there exist chordal graphs $G$ with $c_{\infty}(G) = \Omega(n / \log n)$.
More precisely, it is shown that for every positive integer $m$,
there exists a chordal graph $G$ with $O(m \log m)$ vertices having $c_{\infty}(G) \geq m$.

\begin{definition}[access, accessible]
Say the robber  \emph{has access} to a subset $X\subseteq V(G)$
if there exists a cop-free path from the robber's vertex to a vertex in $X$.
A pair $(X,v)$ with $X\subseteq V(G)$ and $v\in V(G)$ is called \emph{accessible} if
\begin{itemize}
\item $c_{\infty}(G) \geq |X|$,
\item $N(v) = X$, and
\item if there are $|X|-1$ cops in the game, then there exists a strategy for the robber with the following properties:
the robber has access to $X$ in every round, but she never moves to a vertex in $X\cup\{v\}$.
\end{itemize}
\end{definition}

In Figure~\ref{fig:accessible}, $(X_i,v_i)$ is an accessible pair in $G_i$ for $i=1,2$.

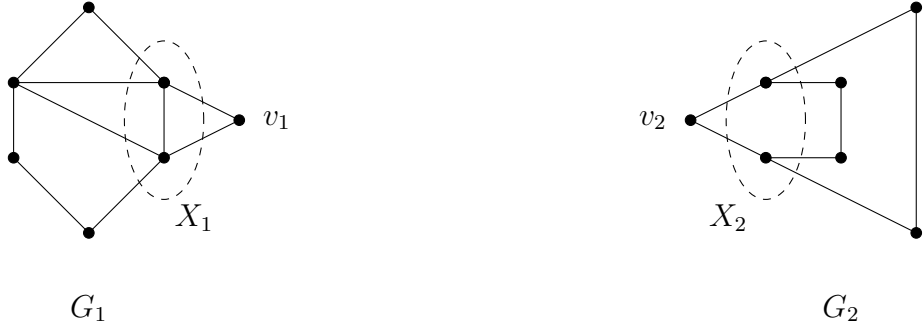
\begin{figure}
\begin{center}
\begin{tikzpicture}[thin]
        \draw (0,0) -- (-1,1) -- (-1,2) -- (0,3) -- (1,2) -- (2,1.5) -- (1,1) -- (1,2) -- (-1,2) -- (1,1) -- (0,0);
		\node at (0,0) [matched]{};
		\node at (-1,1) [matched]{};
		\node at (-1,2) [matched]{};
		\node at (0,3) [matched]{};
		\node at (1,2) [matched]{};
		\node at (2,1.5) [matched]{};
		\node at (1,1) [matched]{};
		\node at (1,2) [matched]{};

        \draw[dashed](1,1.5) ellipse(15pt and 30pt);
        \node at (1.4,0.2) {$X_1$};
        \node at (0,-1) {$G_1$};
        \node at (2.5,1.5) {$v_1$};

        \draw (9,1) -- (10,1) -- (10,2) -- (9,2) -- (11,3) -- (11,0) -- (8,1.5) -- (9,2);
		\node at (9,1) [matched]{};
		\node at (10,1) [matched]{};
		\node at (10,2) [matched]{};
		\node at (9,2) [matched]{};
		\node at (11,3) [matched]{};
		\node at (11,0) [matched]{};
		\node at (8,1.5) [matched]{};
		\node at (9,2) [matched]{};

        \draw[dashed](9,1.5) ellipse(15pt and 30pt);
        \node at (8.5, 0.2) {$X_2$};
        \node at (10,-1) {$G_2$};
        \node at (7.5, 1.5) {$v_2$};

\end{tikzpicture}

\end{center}
\caption{Examples of accessible pairs}
\label{fig:accessible}
\end{figure}

\begin{lemma}
\label{accessible_sum}
Let $G_1,G_2$ be graphs on disjoint vertex sets,
and for $i=1,2$, $(X_i,v_i)$ be an accessible pair in $G_i$ with $|X_i|=k$.
Let $G$ be  a graph with vertex set $V(G) = V_1 \cup U_1 \cup X \cup U_2 \cup V_2 \cup \{v\}$,
and such that
\begin{itemize}
\item For $i=1,2$, $V_i=V(G_i) \setminus \{v_i\}$.
\item We have $|U_1| = |X| = |U_2| = 2|X_1| = 2|X_2|=2k$ and $V_1,U_1,X,U_2,V_2$ are disjoint.
\item The following pairs induce complete bipartite subgraphs of $G$: $(X_1, U_1)$, $(U_1, X)$, $(X, U_2)$, $(U_2, X_2)$.
\item There is no other edge between any two of $V_1,U_1,X,U_2,V_2$, but there can be arbitrary edges inside $U_1,X,U_2$.
\item The set of neighbours of $v$ is precisely $X$.
\end{itemize}
Then $(X,v)$ is an accessible pair in $G$.
\end{lemma}

\begin{figure}
\begin{center}
\begin{tikzpicture}[thin]
        \draw (0,0) -- (-1,1) -- (-1,2) -- (0,3) -- (1,2) -- (1,1) -- (1,2) -- (-1,2) -- (1,1) -- (0,0);
		\node at (0,0) [matched]{};
		\node at (-1,1) [matched]{};
		\node at (-1,2) [matched]{};
		\node at (0,3) [matched]{};
		\node at (1,2) [matched]{};
		\node at (1,1) [matched]{};
		\node at (1,2) [matched]{};

        \draw[dashed](1,1.5) ellipse(15pt and 27pt);
        \node at (0.9,0.2) {$X_1$};
        \node at (0,-1) {$V_1$};

        \draw (9,1) -- (10,1) -- (10,2) -- (9,2) -- (11,3) -- (11,0) -- (8,1.5) -- (9,2);
		\node at (9,1) [matched]{};
		\node at (10,1) [matched]{};
		\node at (10,2) [matched]{};
		\node at (9,2) [matched]{};
		\node at (11,3) [matched]{};
		\node at (11,0) [matched]{};
		\node at (9,2) [matched]{};

        \draw[dashed](9,1.5) ellipse(15pt and 27pt);
        \node at (9.2, 0.2) {$X_2$};
        \node at (10,-1) {$V_2$};

		\node at (3,0) [matched]{};
		\node at (3,1) [matched]{};
		\node at (3,2) [matched]{};
		\node at (3,3) [matched]{};

        \foreach \x in {3,5,7} {
            \foreach \y in {0,1,2,3} {
                \node at (\x, \y) [matched]{};
            }
        }

        \foreach \y in {0,1,2,3} {
            \foreach \z in {0,1,2,3} {
                \draw(3,\y) -- (5,\z);
                \draw(5,\y) -- (7,\z);
                \draw(3,\y) -- (1,1);
                \draw(3,\y) -- (1,2);
                \draw(7,\y) -- (9,1);
                \draw(7,\y) -- (9,2);
                \draw (6,4) .. controls (5,\y - 0.1) .. (5,\y);
            }

        }

        \node at (6,4) [matched] {};
        \node at (6.25,4) {$v$};

        \draw [dashed, rounded corners = 7mm] (-1.8,-0.5) rectangle (2,3.5);
        \foreach \x in {3,5,7} {
            \draw [dashed, rounded corners = 3mm] (\x - 0.6, -0.5) rectangle (\x + 0.6, 3.5);
        }
        \draw [dashed, rounded corners = 7mm] (8,-0.5) rectangle (12,3.5);
        \node at (3,-1) {$U_1$};
        \node at (5,-1) {$X$};
        \node at (7,-1) {$U_2$};

        \draw (3,3) -- (3,2);
        \draw (5,1) -- (5,3);

\end{tikzpicture}

\end{center}
\caption{An example for Lemma~\ref{accessible_sum}}
\label{fig:accessible_sum}
\end{figure}
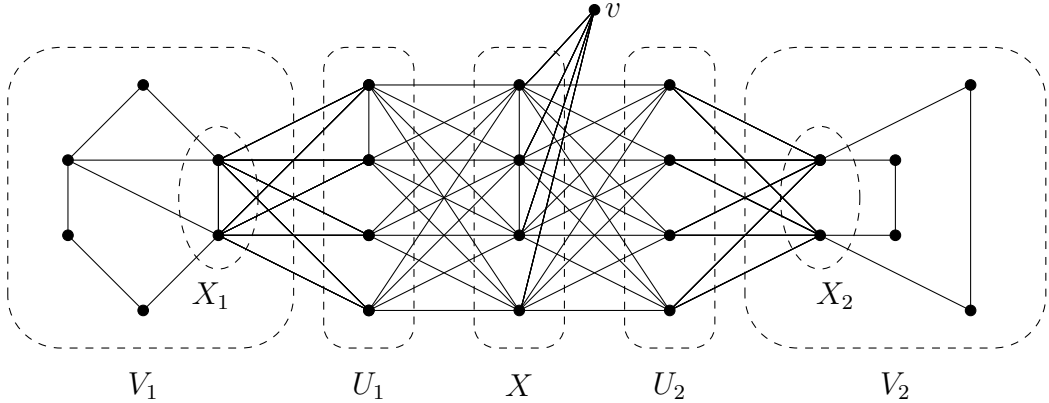

In Figure~\ref{fig:accessible_sum} an example of such a $G$ is given, where $G_1$ and $G_2$ are graphs
shown in Figure~\ref{fig:accessible}.

\begin{proof}
Assume that there are $2k-1$ cops in the game. We prove that the robber has an escaping strategy
that evades the cops forever, and is such that she has access to $X$ in every round, but never moves to $X\cup\{v\}$.
Let $A_i = V_i \cup U_i$ for $i=1,2$.
The strategy has the following invariant: at the end of each round, the robber is at a vertex of $V_j$ for some $1\leq j \leq 2$,
such that there are less than $k$ cops in $A_j$, and the robber has access to $X_j$.
If we provide such a strategy, then since the robber has access to $X_j$ and there are $k$ disjoint paths from $X_j$ to $X$,
the robber has access to $X$ in every round.
We may assume without loss of generality that all the cops start at some vertex in $V_2$, and the robber starts at some vertex in $V_1$,
so the invariant holds at the beginning (with $j=1$).

Assume that the invariant holds at the end of the previous round, say with $j=1$.
This means that the robber is at a vertex of $V_1$, has access to $X_1$, and
there are less than $k$ cops in $A_1$.
In the next round, first the cops move.
If after their move, there are still less than $k$ cops in $A_1$,
then the robber assumes the game is actually played in $G_1$,
where she considers all cops in $V_1$ as they are in $G_1$,
and she considers all cops in $V(G) \setminus V_1$ as if they are at $v_1$;
then she just plays her escaping strategy in $G_1$,
thus she will not go to $X_1\cup\{v_1\}$ and will not be captured in the next round.
Recall that $v_1$ is the vertex in $G_1$ whose set of neighbours is $X_1$.

In the other case, there are at least $k$ cops in $A_1$.
There are at most $k-1$ cops in $A_2$ at this moment, and in particular, at most $k-1$ cops in $V_2$.
Recall that $(X_2,v_2)$ is an accessible pair in $G_2$, which means, in particular, that there exists a vertex $u\in V_2$
such that at this moment there is a cop-free path $P$ from $X_2$ to $u$.
(To see this, note that if one just considers the graph induced by $V_2$ and assumes that the game is played only in this subgraph,
then the robber can choose a vertex that has access to $X_2$.)
Since at the end of the previous round there were less than $k$ cops in $A_1$, there are less than $k$ cops in $V_1$ at this moment.
Hence the robber has access to $X_1$ (note that cops in $V(G) \setminus V_1$ will not block the robber's access to $X_1$),
through which she can pass through $U_1, X, U_2, X_2$ (notice that each of these has at least one cop-free vertex),
and finally go to $u$ along the path $P$.
\end{proof}

It is easy to verify that if both $G_1$ and $G_2$ are chordal graphs and the subgraphs induced by $U_1$, $X$, and $U_2$ are complete graphs,
then the resulting graph $G$ is chordal as well.
This lets us deduce the following lower bound.

\begin{theorem}
\label{chordal_lower}
For every positive integer $m$, there exists a chordal graph $G$ having $O(m \log m)$ vertices and $c_{\infty}(G) \geq m$.
\end{theorem}

\begin{proof}
For every $m$, let $g(m)$ denote the number of vertices of the smallest connected chordal graph that has an accessible pair $(X,v$) with $|X|=m$.
Then, by Lemma~\ref{accessible_sum} and the discussion above,
$$g(2) \leq 7, \quad g(m) \leq 2 (g(\lceil m/2 \rceil)-1) + 6 \lceil m/2 \rceil + 1,$$
which gives $g(m) = O(m \log m)$.
\end{proof}


\section{Cop Number and Treewidth}
\label{chp:treewidth}

A \emph{tree decomposition} of a graph $G$ is a pair $(T,W)$, where $T$ is a tree and
$W = (W_t : t \in V(T))$ is a family of subsets of $V(G)$ such that
\begin{itemize}
\item[(i)] $\bigcup_{t\in V(T)} W_t = V(G)$, and every edge of $G$ has both endpoints in some $W_t$, and
\item[(ii)] For every $v\in V(G)$, the set $\{t : v \in W_t\}$ induces a subtree of $T$.
\end{itemize}
The \emph{width} of $(T,W)$ is
$$\max\{|W_t| - 1 : t \in V(T)\},$$
and the \emph{treewidth} of $G$, written $tw(G)$,
is the minimum width of a tree decomposition of $G$.

We will use the following facts about tree decompositions,
whose proofs can be found in Section~12.3 of the textbook by Diestel~\cite{diestel}.

\begin{proposition}
\label{prop:decompositions}
Let $(T,W)$ be a tree decomposition of a graph $G$.
\begin{itemize}
\item [(a)] Let $A$ be the vertex set of a clique in $G$.
Then there is a $t\in V(T)$ with $A \subseteq W_t$.
\item[(b)] Let $t_1t_2$ be an edge of $T$,
and let $T_1$ and $T_2$ be the components of $T-t_1t_2$,
with $t_1 \in T_1$ and $t_2 \in T_2$.
Define $X = W_{t_1} \cap W_{t_2}$, $U_1 = \cup_{t\in T_1} W_t$ and $U_2 = \cup_{t\in T_2} W_t$.
Then $X$ is a cut-set in $G$,
and there is no edge between $U_1 \setminus X$ and $U_2 \setminus X$.
\end{itemize}
\end{proposition}

For the original Cops and Robber game,
Joret, Kami{\'n}ski, and Theis~\cite{forbidden_subgraphs} proved that for every $G$,
$\frac{tw(G)}{2} + 1$ cops are sufficient to capture the robber.

Write $\Delta=\Delta(G)$ for the maximum degree in $G$.
In this section we prove that for every $G$,
$$\frac{tw(G)+1}{\Delta(G)+1} \leq c_{\infty}(G) \leq tw(G)+1.$$
Moreover, we prove that these bounds are tight.
To prove the lower bound, we relate our Cops and Robber game with another
pursuit-evasion game, called the Helicopter Cops and Robber game.
This game, introduced by Seymour and Thomas~\cite{helicopter},
has two different versions,
and the one we define here is called jump-searching.

\begin{definition}[Helicopter Cops and Robber game (the jump-searching version)]
For $X \subseteq V(G)$, an \emph{$X$-flap} is the vertex set of a connected component of $G-X$.
Two subsets $X,Y \subseteq V(G)$ \emph{touch} if $\overline{N}(X) \cap Y \neq \emptyset$.
A \emph{position} is a pair $(X,R)$, where $X\subseteq V(G)$ and $R$ is an $X$-flap.
($X$ is the set of vertices currently occupied by the cops and $R$ tells us where the robber is
--- since she can run arbitrarily fast, all that matters is which component of $G-X$ contains her.)
At the start, the cops choose a subset $X_0$, and the robber chooses an $X_0$-flap $R_0$.
Note that if there are $k$ cops in the game, then $|X_0| \leq k$.
At the start of round $i$, we have some position $(X_{i-1},R_{i-1})$.
The cops choose a new set $X_i\subseteq V(G)$ with $|X_i| \leq k$ (and no other restriction), and announce it.
Then the robber, knowing $X_i$, chooses an $X_i$-flap $R_i$ which touches $R_{i-1}$.
If this is not possible then the cops have won.
Otherwise, i.e.~if the robber never runs out of valid moves, the robber wins.
\end{definition}

The following lemma establishes a link between the two games.

\begin{lemma}
\label{lem:helicopter_link}
Let $G$ be a graph.
If $k$ cops can capture a robber with unbounded speed in the Cops and Robber game in $G$,
then $k(\Delta+1)$ cops can capture the robber in the Helicopter Cops and Robber game in $G$.
\end{lemma}

\begin{proof}
We consider two games played in two copies of $G$:
the first one, which we call the \emph{real game}, is a game of Helicopter Cops and Robber with $k(\Delta+1)$ cops;
and the second one, the \emph{virtual game}, is the usual Cops and Robber game with $k$ cops and a robber with unbounded speed.
Given a winning strategy for the cops in the virtual game, we need to give a capturing strategy for the cops in the real game.
We translate the moves of the cops from the virtual game to the real game, and translate the moves of the robber from the real game to the virtual game, in such a way that all the translated moves are valid, and if the robber is captured in the virtual game, then she is captured in the real game as well.
Hence, as the cops have a winning strategy in the virtual game, they have a winning strategy in the real game, too.

In the virtual game, initially the cops choose a subset $C_0$ of vertices.
Then the real cops choose $X_0 = \overline{N}(C_0)$.
Recall that $|C_0| \leq k$, so $|X_0| \leq k(\Delta + 1)$.
The real robber chooses $R_0$, which is an $X_0$-flap,
and the virtual robber chooses an arbitrary vertex $r_0 \in R_0$.
In general, at the end of round ${i-1}$ we have $X_{i-1} = \overline{N}(C_{i-1})$ and $r_{i-1} \in R_{i-1}$.

Suppose the virtual robber is not captured in round $i$.
In round $i$, first the virtual cops move to a new set $C_i$.
Each cop either stays still or moves to a neighbour, thus $C_i \subseteq \overline{N}(C_{i-1}) = X_{i-1}$ and
since $R_{i-1}$ was an $X_{i-1}$-flap, $C_i \cap R_{i-1} = \emptyset$.
The real cops choose $X_i = \overline{N}(C_i)$ and announce it.
The real robber, knowing $X_i$, chooses an $X_i$-flap $R_i$ that touches $R_{i-1}$.
If she cannot find a valid move then she is captured and the lemma is proved.
Otherwise, note that by definition $C_i \cap R_i = \emptyset$.
Let $r_i$ be an arbitrary vertex of $R_i$.
The virtual robber moves from $r_{i-1}$ to $r_i$.
Since $R_{i-1}$ and $R_i$ touch, and both of them are connected, $R_{i-1}\cup R_i$ is connected.
Moreover, $C_i$ does not intersect $R_{i-1}\cup R_i$, so this is a valid move in the virtual game.

Now, suppose the virtual robber is captured in round $i$.
We claim that if this happens then the real robber has already been captured in one of the previous rounds.
If this is not the case, then in round $i$, the virtual cops move to a new set $C_i$ such that $r_{i-1} \in C_i$.
Each cop either stays still or moves to a neighbour, thus $C_i \subseteq \overline{N}(C_{i-1}) = X_{i-1}$ and
since $R_{i-1}$ was an $X_{i-1}$-flap, $C_i \cap R_{i-1} = \emptyset$.
But $r_{i-1}\in C_i$ because the virtual robber has been captured in round $i$,
and $r_{i-1}\in R_{i-1}$, thus
$r_{i-1} \in C_i \cap R_{i-1}$, which is a contradiction.
This shows that the real robber will be captured even before the virtual robber, and the proof is complete.
\end{proof}

Seymour and Thomas~\cite{helicopter} proved the following theorem.

\begin{theorem}[\cite{helicopter}]
\label{thm:helicopter}
The minimum number of cops needed to capture a robber in Helicopter Cops and Robber game is equal to the treewidth of the graph plus one.
\end{theorem}

Using this we have the following.

\begin{theorem}
\label{thm:treewidth_bound}
For every graph $G$ we have
$$\frac{tw(G)+1}{\Delta(G)+1} \leq c_{\infty}(G) \leq tw(G)+1,$$
and these bounds are tight.
\end{theorem}

\begin{proof}
The lower bound follows from Lemma~\ref{lem:helicopter_link} and Theorem~\ref{thm:helicopter}.
To prove tightness of the lower bound, let $G$ be the complete graph on $n$ vertices.
Then
it follows from part (a) of Proposition~\ref{prop:decompositions}
that $G$ has treewidth $n-1$.
A single cop can capture the robber in $G$, since $G$ has domination number one.
Hence, the complete graph on $n$ vertices
has treewidth $n-1$, maximum degree $n-1$, and cop number 1, so the lower bound is tight.

Now we prove the upper bound. Consider a tree decomposition $(T,W)$ of $G$ having minimum width.
Assume that there are $tw(G)+1$ cops in the game,
so for every $t\in V(T)$, there are at least $|W_t|$ cops in the game.
The cops start at $W_{t_1}$ for some arbitrary $t_1\in V(T)$.
Assume that the robber starts at $r_0$, and let $t$ be such that $r_0 \in W_{t}$.
Let $t_2$ be the neighbour of $t_1$  in the unique $(t_1,t)$-path in $T$.
Let $T_1$ and $T_2$ be the components of $T-t_1t_2$,
with $t_1 \in T_1$ and $t_2 \in T_2$.
Define $X = W_{t_1} \cap W_{t_2}$, $U_1 = \cup_{t\in T_1} W_t$, and $U_2 = \cup_{t\in T_2} W_t$.
So the cops are all in $U_1$ and the robber is at a vertex in $U_2 \setminus X$.
Note that the number of cops is at least $|W_{t_2}|$.
Now the cops move in order to occupy $W_{t_2}$,
in such a way that the cops in $X$ stay still.
After some rounds, the cops will be located at $W_{t_2}$,
and during those rounds the robber could not escape from $U_2 \setminus X$,
because by part (b) of Proposition~\ref{prop:decompositions},
there is no edge between $U_1 \setminus X$ and $U_2 \setminus X$.
When the cops have established in $W_{t_2}$, the total space available to the robber has been decreased.
Continuing similarly the cops will eventually capture the robber.

%

Next we prove that the upper bound is tight.
Let $m\geq 4$ be a positive integer.
Define graph $G$ 
as follows.
$G$ has a total of $m + 2m\binom{m}{2}$ vertices,
with a certain independent set $\{v_1,\dots,v_m\}$,
such that every two of the $v_i$'s are connected by $m$ disjoint paths of length 3,
and $G$ does not have any other edge.
Thus $G$ has a total of $3m\binom{m}{2}$ edges.
We show that there exists a tree decomposition of $G$ with width $\max\{m-1, 3\}$.
Let $T$ be the star with $1 + m\binom{m}{2}$ vertices, and let $r$ be its dominating vertex.
Define $W_r = \{v_1,\dots,v_m\}$.
To each path $v_i u_1 u_2 v_j$ assign a leaf $l$ of the tree and set $W_l = \{v_i, u_1, u_2, v_j\}$.
It is easy to verify that $(T,W)$ is a tree decomposition of $G$ with width $\max\{m-1, 3\}$.
Note that $m \geq 4$, so $tw(G) \leq m-1$.

Now we show that $c_{\infty}(G) \geq m$, which completes the proof.
It suffices to show that $m-1$ cops cannot capture a robber with unbounded speed.
Say a cop \emph{controls} a vertex $u$ if the cop is at $u$ or at an adjacent vertex.
If there are $m-1$ cops in the game, we show that the robber can play such that at the end of each round,
if the cops are in $C \subseteq V(G)$, then the robber is at a vertex $r\in \{v_1,\dots,v_m\} \setminus \overline{N}(C)$.
The robber can choose such a vertex at the beginning,
because the distance between any two of the $v_i$'s is 3, so
each cop can control at most one of the $v_i$'s.
Assume that at the end of round $i$ the cops are in $C_i$ and the robber is at $r_i\in \{v_1,\dots,v_m\} \setminus \overline{N}(C_i)$.
In round $i+1$, first the cops move to $C_{i+1} \subseteq \overline{N}(C_i)$. So the robber is not captured.
There exists a vertex $r_{i+1}\in \{v_1,\dots,v_m\} \setminus \overline{N}(C_{i+1})$,
because every cop controls at most one of the $v_i$'s.
If $r_{i+1}=r_i$ then the robber does not move at all.
Otherwise, there are $m$ disjoint $(r_i,r_{i+1})$-paths in $G$, and $m-1$ cops, so at least one of these paths is cop-free,
and the robber moves along that path to $r_{i+1}$.
\end{proof}


\section{Planar Graphs}
\label{chp:planar}

In one of the first papers
on the original Cops and Robber game,
Aigner and Fromme~\cite{speed_one_planar}
proved that three cops can capture the robber in any planar graph.
In this section, we show that if $G$ is planar then $c_{\infty}(G) = \Theta(tw(G))$.
This proves that every planar graph $G$ has $c_{\infty}(G) = O(\sqrt n)$,
and also gives an $O(1)$-approximation algorithm for finding the cop number of a planar graph.
These results hold also when $G$ does not contain any fixed apex graph as a minor.

An \emph{apex graph} is a graph $H$ that has a vertex $v$ such that $H-v$ is planar.
For example, $K_5$ is an apex graph.
The following theorem was proved in a weaker form by
Demaine, Fomin, Hajiaghayi, and Thilikos~\cite{bidimensionality_1},
and then in its current form by
Demaine and Hajiaghayi~\cite{bidimensionality}.

\begin{theorem}[\cite{bidimensionality_1, bidimensionality}]
\label{useful}
Let $H$ be a fixed apex graph.
There is a constant $C_H$ such that the following holds.
Let $g:\mathbb{N}\rightarrow\mathbb{N}$ be a strictly increasing function,
and $P(G)$ be a graph parameter with the following two properties.
\begin{enumerate}
\item
If $G$ is the $r\times r$ grid augmented with additional edges such that each vertex
is incident to $C_H$ edges to non-boundary vertices of the grid, then $P(G) \geq g(r)$.
\item
$P(G)$ does not increase by contracting an edge of $G$.
\end{enumerate}
Then, for any graph $G$ that does not contain $H$ as a minor,
the treewidth of $G$ is $O\left(g^{-1}(P(G))\right)$.
\end{theorem}


\begin{theorem}
\label{thm:planar_treewidth}
Let $H$ be a fixed apex graph.
Any graph $G$ that does not contain $H$ as a minor has $c_{\infty}(G) = \Theta(tw(G))$.
In particular, if $G$ is planar then $c_{\infty}(G) = \Theta(tw(G))$.
\end{theorem}

\begin{proof}
We show that the parameter $c_{\infty}(G)$  satisfies the two properties given in Theorem~\ref{useful}, with $g(r) = (r+1)/(5+C_H)$.
First, an augmented $r\times r$ grid has treewidth $r$ and maximum degree at most $4+C_H$,
so by Theorem~\ref{thm:treewidth_bound} its cop number is at least $(r+1) / (5+C_H)$.

Second, we need to show that the cop number does not increase by contracting an edge.
It is not difficult to show that contracting an edge does not help the robber,
since she has unbounded speed, and it does not hurt the cops.
Therefore, contracting an edge does not increase the cop number.

Therefore, by  Theorem~\ref{useful}, if $G$ does not contain $H$ as a minor,
then $tw(G) = O(c_{\infty}(G))$.
By Theorem~\ref{thm:treewidth_bound}, $c_{\infty}(G) \leq tw(G)+1$, so we have
$c_{\infty}(G) = \Theta(tw(G))$.
To get the second statement, note that a planar graph does not contain $K_5$ as a minor.
\end{proof}

\begin{corollary}
Let $H$ be a fixed apex graph.
Any graph $G$ that does not contain $H$ as a minor has $c_{\infty}(G) = O(\sqrt n)$,
and this bound is tight.
In particular, any planar graph $G$ has $c_{\infty}(G) = O(\sqrt n)$, and this bound is tight.
\end{corollary}

\begin{proof}
It is known (see, e.g., \cite{minor_free_treewidth}) that if $G$ does not have $H$ as a minor, then $tw(G) = O(\sqrt n)$.
The $m\times m$ grid has $m^2$ vertices and by Theorem~\ref{thm:treewidth_bound}, its cop number
is at least $(m+1)/5$.
Hence the bound is tight.
\end{proof}

\begin{corollary}
Let $H$ be a fixed apex graph.
There is a constant-factor approximation algorithm for computing the cop number of a graph
that does not contain $H$ as a minor.
In particular, There is a constant-factor approximation algorithm for computing the cop number of a planar graph.
\end{corollary}

\begin{proof}
Feige, Hajiaghayi, and Lee~\cite{approximation_treewidth_minor_free}
have developed an $O(1)$-approximation algorithm for finding the treewidth of a graph that does not contain $H$ as a minor.
\end{proof}


\section{Cartesian Products of Complete Graphs, and Hypercube Graphs}
\label{chp:hypercube}
Let $G_1,G_2,\dots,G_m$ be graphs.
Define $G$ to be the graph with vertex set $V(G_1)\times V(G_2)\times \dots \times V(G_m)$
with vertices $(u_1,u_2,\dots,u_m)$ and $(v_1,v_2,\dots,v_m)$ being adjacent if there exists an index $1\leq j \leq m$ such that
\begin{itemize}
\item $u_i = v_i$ for all $i\neq j$, and
\item $u_j$ and $v_j$ are adjacent in $G_j$.
\end{itemize}
Then $G$ is called the \emph{Cartesian product} of $G_1,G_2,\dots,G_m$.
If every $G_i$ is isomorphic to an edge, then the graph $G$ is called the \emph{$m$-hypercube} graph and denoted by $\mathcal {H}_m$.
In this section we give bounds for the Cartesian product of complete graphs with the same size,
and tighter bounds for hypercube graphs.
Neufeld and Nowakowski~\cite{speed_one_products} have studied the original Cops and Robber game
played on products of graphs.
They have determined exactly the number of cops needed to capture the robber,
when $G$ is the Cartesian product of complete graphs with not necessarily the same size,
and when $G$ is the Cartesian product of an arbitrary number of trees and cycles.

First, we prove an easy lemma, which gives a weak upper bound for the cop number
of the Cartesian product of graphs.

\begin{lemma}
\label{lem:bounds_products}
Let $G_1,G_2,\dots,G_m$ be graphs and let $n_i$ denote the number of vertices of $G_i$ for $1\leq i\leq m$.
Let $G$ be the Cartesian product of $G_1,G_2,\dots,G_m$, and $n = |V(G)| = n_1 n_2 \dots n_m$. Then we have
$$ c_{\infty}(G) \leq \frac{nc_{\infty}(G_1)}{n_1}.$$
\end{lemma}

\begin{proof}
We give a strategy for $nc_{\infty}(G_1) / n_1$ cops to capture the robber in $G$.
Let $k = c_{\infty}(G_1)$.
By definition, there is a winning strategy for $k$ cops when the game is played in $G_1$.
We consider a \emph{virtual game}, in which $k$ virtual cops are capturing a {virtual robber} in $G_1$.
(Using a virtual game for bounding the cop number is also used in the proof of Lemma~\ref{lem:helicopter_link},
where it has been explained in more detail.)
For every virtual cop, we put $n/n_1 = n_2 n_3 \dots n_m$ real cops in the real game,
such that if the virtual cop is in $u_1 \in V(G_1)$, then the real cops occupy $\{u_1\} \times V(G_2) \times \dots \times V(G_m)$.
Also, if the real robber is at $(v_1,\dots,v_m) \in G$, then the virtual robber is at $v_1\in G_1$.
It is not hard to see that the real cops can move in such a way that these constraints hold throughout the games.
Hence, once the virtual robber has been captured, the real robber has also been captured, and the proof is complete.
\end{proof}

\begin{theorem}
\label{thm:bounds_products}
Let $G_1,G_2,\dots,G_m$ be graphs, and
let $G$ be the Cartesian product of $G_1,G_2,\dots,G_m$, and $n = |V(G)|$.
Then we have
\begin{itemize}
\item[(a)]
There exist positive constants $\kappa_1,\kappa_2$
such that
if every $G_i$ is isomorphic to the complete graph on $k$ vertices, then
$$\frac{\kappa_1 n}{km\sqrt m } \leq c_{\infty}(G) \leq \min\left\{ \frac{n}{k},\frac{\kappa_2 n}{\sqrt m}\right\}.$$
\item[(b)]
If every $G_i$ is isomorphic to an edge, i.e.~if $G$ is the $m$-hypercube $\mathcal{H}_m$,
then there exist constants $\eta_1,\eta_2>0$ such that
$$\frac{\eta_1 n}{m\sqrt m} \leq c_{\infty}(G) \leq \frac{\eta_2 n}{m}.$$
\end{itemize}
\end{theorem}

\begin{proof}

\begin{itemize}
\item[(a)]
Sunil Chandran and Kavitha~\cite{hypercube_treewidth} have proved that
$$tw(G) = \Theta\left(\frac{n}{\sqrt m}\right).$$
As $G$ has maximum degree $O(mk)$,
the lower bound follows from Theorem~\ref{thm:treewidth_bound}.
The upper bound $c_{\infty}(G) = O\left(n /\sqrt m\right)$ follows from the same theorem,
and the bound $c_{\infty}(G) \leq n / k$ follows from Lemma~\ref{lem:bounds_products},
since $G_1$ is a complete graph and has $c_{\infty}(G_1) = 1$.

\item[(b)]
We claim that for any positive $m$, the $m$-hypercube $\mathcal{H}_m$ has  domination number at most $2^{m+1} / (m+1)$.
Indeed, if for some positive integer $k$, $m=2^k-1$, then it is well known that $\mathcal{H}_m$ has  domination number exactly $2^m / (m+1)$ (see~\cite{hypercube_dominating_number} for example).
Otherwise, let $k$ be the largest integer with $2^k - 1 \leq m$. Thus $m < 2^{k+1}-1$.
It is easy to see that for every graph $G$ with domination number $r$, the Cartesian product of $G$ and an edge has domination number at most $2r$.
Hence one can prove using induction that for $i\geq 2^k-1$,
the domination number of $\mathcal{H}_i$ is at most
$$\frac{2^{2^k - 1}}{2^k}2 ^{i - (2^k - 1)} = 2^{i-k}.$$
In particular, the domination number of $\mathcal{H}_m$ is at most $2^{m - k} < \frac{2^{m+1}}{m+1}$.

The upper bound follows from the above claim (recall that $n=2^m)$,
and the fact that the domination number is always an upper bound for the cop number.

Sunil Chandran and Kavitha~\cite{hypercube_treewidth} have proved that $tw(\mathcal{H}_m )= \Theta(2^m / \sqrt m)$.
Since $\mathcal{H}_m$ has maximum degree $m$, the lower bound follows from Theorem~\ref{thm:treewidth_bound}.\qedhere
\end{itemize}

\end{proof}


\section{Existence of Graphs with Linear Cop Number}
\label{chp:large}
Theorem~\ref{thm:treewidth_bound} is especially useful for giving lower bounds for the cop number,
when the graph has small maximum degree. To illustrate this, we use it to give a short proof for
the fact that for each $n$, there exists a connected graph on $n$ vertices with cop number $\Theta(n)$,
which is proved by Frieze~at~al.~\cite{variations} using other ideas.

\begin{theorem}
For each $n$, there exists a connected graph on $n$ vertices with cop number $\Theta(n)$.
\end{theorem}

\begin{proof}
Let $G$ be an {E}rd\"{o}s-R\'{e}nyi random graph with $n$ vertices and $2n$ edges.
Kloks~\cite{random_graphs_treewidth} has proved that there is a positive constant $\beta$ such that
we have $tw(G) > \beta n$ with probability approaching one,
as $n$ goes to infinity.

Each vertex of $G$ has average degree $2|E(G)|/|V(G)| = 4$.
Hence by Markov's inequality, the probability that a fixed vertex has degree larger than $16/\beta$ is less than $\beta/4$.
By linearity of expectation, the expected number of vertices of degree larger than $16/\beta$ is less than $n\beta / 4$.
Therefore by Markov's inequality, with probability at least $1/2$, $G$ has at most $n \beta / 2$ vertices of degree larger than $16/\beta$.

Consequently, for $n$ large enough, there exists a graph $G_n$ such that
\begin{itemize}
\item $tw(G_n) > \beta n$, and
\item $G_n$ has at most $n \beta / 2$ vertices of degree larger than $16/\beta$.
\end{itemize}
Let $H_n$ denote the graph obtained from $G_n$ by deleting all vertices of degree larger than $16/\beta$.
Deleting each vertex does not decrease treewidth by more than 1.
Thus we have $$ |V(H_n)| \leq n, \qquad tw(H_n) \geq n \beta / 2,\ \mathrm{and} \qquad \Delta(H_n) \leq 16/\beta.$$
By Theorem~\ref{thm:treewidth_bound},
$$c_{\infty}(H_n) \geq \frac{tw(H_n) + 1}{\Delta(H_n)+1} \geq \frac{tw(H_n)}{2\Delta(H_n)},$$ and so
$$ \frac {|V(H_n)|}{c_{\infty}(H_n)} \leq \frac{2 |V(H_n)| \Delta(H_n)}{tw(H_n)} \leq 2 n \times \frac{16}{\beta} \times \frac{2}{n\beta} =64/\beta^2 = O(1),$$
completing the proof.
\end{proof}


\section{Open Problems}
\label{chp:future}

In this section we present a few open questions and research directions on this variant of the Cops and Robber game.

\begin{enumerate}
%
\item
Fomin~et~al.~\cite{fast_robber_first} asked about the complexity of computing $c_{\infty}(G)$ when $G$ is an interval graph.
We proved that this problem is 3-approximable (see Theorem~\ref{thm:intervals_approximation}), but it is still not known if it is NP-hard or not.

\item
We proved that there exist chordal graphs $G$ with $c_{\infty}(G) = \Omega(n / \log n)$ (see Theorem~\ref{chordal_lower}).
Are there  chordal graphs $G$ with $c_{\infty}(G) = \Omega(n)$ ?

\item
Let $H$ be a fixed apex graph.
In Theorem~\ref{thm:planar_treewidth} we proved that if $G$ does not have $H$ as a minor,
then $c_{\infty}(G) = \Theta(tw(G))$.
Is this result true when $H$ is a general graph?

\item
In part (b) of Theorem~\ref{thm:bounds_products}, we have determined the cop number of
the $m$-hypercube graph up to an $O(\sqrt m)$ factor.
What is the exact value?

\item
Fomin~et~al.~\cite{fast_robber_first} proved that computing $c_{\infty}(G)$ is NP-hard.
Is this problem in NP?
To show that this problem is in NP, one needs to prove that there is always an efficient
way to describe the cops' strategy. This has been done for the Helicopter Cops and Robber game~\cite{helicopter}.
\end{enumerate}

\noindent\textbf{Acknowledgement.}
The author is grateful to Nick~Wormald for continuous support and lots of fruitful discussions.

\bibliographystyle{amsplain}
\bibliography{graphsearching}

\providecommand{\bysame}{\leavevmode\hbox to3em{\hrulefill}\thinspace}
\providecommand{\MR}{\relax\ifhmode\unskip\space\fi MR }
\providecommand{\MRhref}[2]{%
  \href{http://www.ams.org/mathscinet-getitem?mr=#1}{#2}
}
\providecommand{\href}[2]{#2}
\begin{thebibliography}{10}

\bibitem{speed_one_planar}
M.~Aigner and M.~Fromme, \emph{A game of cops and robbers}, Discrete Appl.
  Math. \textbf{8} (1984), no.~1, 1--11. \MR{739593 (85f:90124)}

\bibitem{minor_free_treewidth}
N.~Alon, P.~Seymour, and R.~Thomas, \emph{A separator theorem for nonplanar
  graphs}, J. Amer. Math. Soc. \textbf{3} (1990), no.~4, 801--808. \MR{1065053
  (91h:05076)}

\bibitem{shooting_cop}
A.~Bonato, E.~Chiniforooshan, and P.~Pra{\l}at, \emph{Cops and robbers from a
  distance}, Theor. Comput. Sci. \textbf{411} (2010), 3834--3844.

\bibitem{hypercube_treewidth}
L.~Sunil Chandran and T.~Kavitha, \emph{The treewidth and pathwidth of
  hypercubes}, Discrete Math. \textbf{306} (2006), no.~3, 359--365. \MR{2204113
  (2006k:68135)}

\bibitem{regular_visible_robber}
N.~E. Clarke, \emph{A witness version of the cops and robber game}, Discrete
  Math. \textbf{309} (2009), no.~10, 3292--3298. \MR{2526747 (2011a:91080)}

\bibitem{bidimensionality_1}
E.~D. Demaine, F.~V. Fomin, M.~Hajiaghayi, and D.~M. Thilikos,
  \emph{Bidimensional parameters and local treewidth}, SIAM J. Discrete Math.
  \textbf{18} (2004/05), no.~3, 501--511 (electronic). \MR{2134412
  (2006c:05128)}

\bibitem{bidimensionality}
E.~D. Demaine and M.~Hajiaghayi, \emph{Linearity of grid minors in treewidth
  with applications through bidimensionality}, Combinatorica \textbf{28}
  (2008), no.~1, 19--36. \MR{2399006 (2009g:05159)}

\bibitem{diestel}
R.~Diestel, \emph{Graph theory}, third ed., Graduate Texts in Mathematics, vol.
  173, Springer-Verlag, Berlin, 2005. \MR{2159259 (2006e:05001)}

\bibitem{finding_connectivity}
S.~Even and R.~E. Tarjan, \emph{Network flow and testing graph connectivity},
  SIAM J. Comput. \textbf{4} (1975), no.~4, 507--518. \MR{0436964 (55 \#9898)}

\bibitem{approximation_treewidth_minor_free}
U.~Feige, M.~Hajiaghayi, and J.~R. Lee, \emph{Improved approximation algorithms
  for minimum weight vertex separators}, SIAM J. Comput. \textbf{38} (2008),
  no.~2, 629--657. \MR{2411037 (2009g:68267)}

\bibitem{fast_robber_first}
F.~Fomin, P.~Golovach, and J.~Kratochv\'{\i}l, \emph{On tractability of cops
  and robbers game}, Fifth Ifip International Conference On Theoretical
  Computer Science – Tcs 2008 (Giorgio Ausiello, Juhani Karhum{\"a}ki,
  Giancarlo Mauri, and Luke Ong, eds.), IFIP International Federation for
  Information Processing, vol. 273, Springer Boston, 2008, pp.~171--185.

\bibitem{fast_robber_first_journal}
F.~V. Fomin, P.~A. Golovach, J.~Kratochv{\'{\i}}l, N.~Nisse, and K.~Suchan,
  \emph{Pursuing a fast robber on a graph}, Theoret. Comput. Sci. \textbf{411}
  (2010), no.~7-9, 1167--1181. \MR{2606052}

\bibitem{large_girth}
P.~Frankl, \emph{Cops and robbers in graphs with large girth and {C}ayley
  graphs}, Discrete Appl. Math. \textbf{17} (1987), no.~3, 301--305. \MR{890640
  (88f:90204)}

\bibitem{variations}
A.~Frieze, M.~Krivelevich, and P.~Loh, \emph{Variations on cops and robbers},
  J. Graph Theory, to appear.

\bibitem{survey_hahn}
G.~Hahn, \emph{Cops, robbers and graphs}, Tatra Mt. Math. Publ. \textbf{36}
  (2007), 163--176. \MR{2378748 (2009b:05254)}

\bibitem{interval_detection}
W.~L. Hsu, \emph{A simple test for interval graphs}, Graph-theoretic concepts
  in computer science ({W}iesbaden-{N}aurod, 1992), Lecture Notes in Comput.
  Sci., vol. 657, Springer, Berlin, 1993, pp.~11--16. \MR{1244121 (94f:05140)}

\bibitem{randomized_local_visibility}
V.~Isler, S.~Kannan, and S.~Khanna, \emph{Randomized pursuit-evasion with local
  visibility}, SIAM J. Discrete Math. \textbf{20} (2006), no.~1, 26--41
  (electronic). \MR{2257242 (2007h:91038)}

\bibitem{forbidden_subgraphs}
G.~Joret, M.~Kami{\'n}ski, and D.~O. Theis, \emph{The cops and robber game on
  graphs with forbidden (induced) subgraphs}, Contrib. Discrete Math.
  \textbf{5} (2010), no.~2, 40--51 (electronic).

\bibitem{random_graphs_treewidth}
T.~Kloks, \emph{Treewidth}, Lecture Notes in Computer Science, vol. 842,
  Springer-Verlag, Berlin, 1994, Computations and approximations. \MR{1312164
  (96d:05038)}

\bibitem{lu_peng}
L.~Lu and X.~Peng, \emph{On {M}eyniel's conjecture of the cop number},
  submitted, 2009.

\bibitem{speed_one_products}
S.~Neufeld and R.~Nowakowski, \emph{A game of cops and robbers played on
  products of graphs}, Discrete Math. \textbf{186} (1998), no.~1-3, 253--268.
  \MR{1623932 (99k:90233)}

\bibitem{game_definition_1}
R.~Nowakowski and P.~Winkler, \emph{Vertex-to-vertex pursuit in a graph},
  Discrete Math. \textbf{43} (1983), no.~2-3, 235--239. \MR{685631 (84d:05138)}

\bibitem{hypercube_dominating_number}
D.~Peleg and J.~D. Ullman, \emph{An optimal synchronizer for the hypercube},
  SIAM J. Comput. \textbf{18} (1989), no.~4, 740--747. \MR{1004794 (91a:68214)}

\bibitem{game_definition_2}
A.~Quilliot, \emph{Jeux et pointes fixes sur les graphes}, Ph.D. thesis,
  Universit\'e de Paris VI, 1978.

\bibitem{scott_sudakov}
A.~Scott and B.~Sudakov, \emph{A new bound for the cops and robbers problem},
  {\tt arXiv:1004.2010v1 [math.CO]}.

\bibitem{helicopter}
P.~D. Seymour and R.~Thomas, \emph{Graph searching and a min-max theorem for
  tree-width}, J. Combin. Theory Ser. B \textbf{58} (1993), no.~1, 22--33.
  \MR{1214888 (94b:05197)}

\end{thebibliography}

\end{document}